\documentclass{amsart}
\usepackage{amssymb,amsmath,amsthm,amsfonts}
\usepackage{mathtools}
\usepackage{bm}
\usepackage{microtype}
\usepackage{enumitem}
\setlist{nosep, leftmargin=1.5em}
\usepackage{tikz}
\usepackage{tikz-cd}
\usepackage{xcolor}
\definecolor{linkgreen}{rgb}{0,0.5,0}
\usepackage[colorlinks=true,allcolors=blue]{hyperref}
\usepackage[capitalize,noabbrev]{cleveref}
\usepackage{aliascnt}

\newcommand{\newaliastheorem}[3]{%
  \newaliascnt{#1}{theorem}%
  \newtheorem{#1}[#1]{#2}%
  \aliascntresetthe{#1}%
  \crefname{#1}{#2}{#3}%
  \Crefname{#1}{#2}{#3}%
}

\theoremstyle{plain}
\newtheorem{theorem}{Theorem}[section]
\newaliastheorem{proposition}{Proposition}{Propositions}
\newaliastheorem{lemma}{Lemma}{Lemmas}
\newaliastheorem{corollary}{Corollary}{Corollaries}
\newaliastheorem{conjecture}{Conjecture}{Conjectures}

\theoremstyle{definition}
\newaliastheorem{definition}{Definition}{Definitions}
\newaliastheorem{example}{Example}{Examples}
\newaliastheorem{notation}{Notation}{Notations}
\newaliastheorem{convention}{Convention}{Conventions}
\newaliastheorem{question}{Question}{Questions}

\theoremstyle{remark}
\newaliastheorem{remark}{Remark}{Remarks}
\newaliastheorem{observation}{Observation}{Observations}

\crefname{theorem}{Theorem}{Theorems}
\Crefname{theorem}{Theorem}{Theorems}
\crefname{section}{Section}{Sections}
\crefname{subsection}{Subsection}{Subsections}
\Crefname{section}{Section}{Sections}
\Crefname{subsection}{Subsection}{Subsections}

\newcommand{\1}{\mathbf 1}

\newcommand{\Z}{\mathbb{Z}}

\newcommand{\M}{\mathcal M}

\newcommand{\V}{\mathcal V}

\usepackage[backend=biber,style=alphabetic]{biblatex}
\usepackage[displaymath, mathlines]{lineno}

\title{Level Totients for Integer Mosaics}

\author{The Anh Dao}
\address{\newline Concordia International School Hanoi, Hanoi, 100000, Vietnam}
\email{theanh.dao@concordiahanoi.org}

\author{Tim Kasian}
\address{\newline The Taft School, Watertown, CT 06795, USA}
\email{tkasian@taftschool.org}

\author{Mostafa Mirabi}
\address{\newline The Taft School, Watertown, CT 06795, USA and \newline  Wesleyan University, Middletown, CT 06459, USA}
\email{mmirabi@wesleyan.edu}
\urladdr{https://sites.google.com/site/mostafamirabi}

\date{}

\begin{document}

\begin{abstract}
We study a level analog of Euler's totient function for integer mosaics. Let $P_i(n)$ be the set of primes appearing in the first $i$ levels of the mosaic of $n$, and let $\varphi_i(n)$ count the integers $m\leq n$ for which $P_i(m)\cap P_i(n)=\varnothing$. We prove a M\"obius divisor-sum formula for $\varphi_i(n)$ and reduce it to a sum over a set $\V_{i,S}$ of powerful integers. If $S\neq\varnothing$, $i\geq 2$, and $q=\min S$, then
\[
        |\V_{i,S}\cap[1,x]|\sim C_{i,S}x^{1/q},
\]
with $C_{i,S}$ an explicit positive Euler-product constant. For fixed $S$, the density $\delta_i(S)$ of integers whose first $i$ levels avoid $S$ exists and has an Euler product, and the number of such integers up to $N$ is $\delta_i(S)N+O_{i,S}(N^{1/q})$.
Taking $S=P_i(n)$ gives
\[
        \frac{\varphi_i(n)}{n}=\delta_i(P_i(n))+O_\varepsilon(n^{-1/2+\varepsilon})
\]
uniformly in $n$.
\end{abstract}

\maketitle


\section{Introduction}
\label{sec:Intro}

An \emph{integer mosaic} is obtained by repeatedly applying the fundamental theorem of arithmetic to the exponents in the prime factorization of a positive integer. Mullin introduced this construction in \cite{Mullin1963Models,Mullin1963Functions}. Subsequent work studied arithmetic functions on mosaics, generalized multiplicativity, and mosaic divisors \cite{Mullin1967Mobius,MR2501345}. The same recursive construction is also called the \emph{prime tower factorization} in recent work on height, prime occurrence, and tree shape \cite{DevlinGnang,DeKoninckVerreault,ContiContucciIudelevich2024,ContiContucciIudelevich2025}.

If $ n=\prod_{j=1}^k p_j^{\alpha_j}$
is the canonical prime factorization of $n>1$, then the mosaic of $n$ is defined recursively by
\[
        \M(1)=1,\qquad\M(n)=\prod_{j=1}^k p_j^{\M(\alpha_j)}.
\]
For example,
\[
        995328=2^{12}3^5\quad\text{gives}\quad
        \M(995328)=2^{\M(12)}3^{\M(5)}=2^{2^2 3}3^5.
\]
The first level consists of the ordinary prime divisors of $n$, the second level consists of the prime divisors of the exponents, and later levels are obtained by continuing this process. We write $P_i(n)$ for the set of primes appearing in the first $i$ levels. Thus $P_1(n)$ is the usual set of prime divisors, while $P_2(n)$ also records primes dividing the exponents.

Gillman introduced the corresponding notion of $i$-relative primality \cite{Gillman1990,Gillman1992}: two positive integers $a$ and $b$ are $i$-relatively prime if $ P_i(a)\cap P_i(b)=\varnothing.$
The associated level totient is
\[
        \varphi_i(n)=\#\{m\leq n:P_i(m)\cap P_i(n)=\varnothing\}.
\]
For $i=1$, this is Euler's totient function. For $i\geq 2$, the condition also uses primes that occur in the exponents. For instance,
\[
        P_2(8)=P_2(2^3)=\{2,3\},\qquad P_2(9)=P_2(3^2)=\{2,3\},
\]
so $8$ and $9$ are ordinarily coprime but not $2$-relatively prime. Thus the level totient differs from Euler's function already at level $2$.

Bildhauser, Erickson, Tacoma, and Gillman introduced the functions $\varphi_i$ in their study of integer mosaics and asked for a method to compute them \cite{MR2501345}. We compute $\varphi_i(n)$ by introducing the indicator 
\[
        \chi_{i,S}(m)=\mathbf 1\{P_i(m)\cap S=\varnothing\},
\]
where $S$ is a finite set of primes.
Since $P_i(ab)=P_i(a)\cup P_i(b)$ whenever $\gcd(a,b)=1$, the function
$\chi_{i,S}$ is multiplicative. Its M\"obius transform
\[
        h_{i,S}=\chi_{i,S}*\mu
\]
therefore gives a divisor formula for the counting function associated to $S$.
Taking $S=P_i(n)$ gives the formula for $\varphi_i(n)$.

We next note that the prime-power values of $h_{i,S}$ are explicit. If $p\in S$, then
\[
        h_{i,S}(p)=-1,\qquad h_{i,S}(p^a)=0\quad(a\ge2).
\]
If $p\notin S$, then $h_{i,S}(p)=0$, while for $a\geq2$,
\[
        h_{i,S}(p^a)=
        \chi_{i-1,S}(a)-\chi_{i-1,S}(a-1).
\]
Thus primes in $S$ that contribute nonzero terms occur only to the first power, while primes outside $S$ can occur only with exponents $a\geq2$ at which $\chi_{i-1,S}(a)$ and $\chi_{i-1,S}(a-1)$ differ.

This motivates the definition
\[
        E_{i,S}=\{a\geq2:\chi_{i-1,S}(a)\neq\chi_{i-1,S}(a-1)\}.
\]
We define $\V_{i,S}$ to be the set of integers whose prime factors lie outside $S$, and whose nonzero prime exponents lie in $E_{i,S}$. Every element of $\V_{i,S}$ is powerful. After the contribution from primes in $S$ is summed separately, the divisor formula becomes a sum over $v\in\V_{i,S}$.

When $S\neq \varnothing$ and $i\geq 2$, the least element of $E_{i,S}$ is $q=\min S$. We prove the asymptotic
\[
        |\V_{i,S}\cap[1,x]|\sim C_{i,S}x^{1/q},
\]
where $C_{i,S}>0$ is given by an explicit Euler product. 

For fixed $S$, we prove that the density
\[
        \delta_i(S)=\lim_{N\to\infty}\frac1N\#\{m\leq N:P_i(m)\cap S=\varnothing\}
\]
exists and has an Euler product. The number of integers up to $N$ whose first $i$ levels avoid $S$ equals
\[
    \delta_i(S)N+O_{i,S}(N^{1/q}).
\] 
Taking $S=P_i(n)$ and using the bound $\prod_{p\in P_i(n)}p\leq n$ gives the uniform estimate
\[
    \frac{\varphi_i(n)}n=\delta_i(P_i(n))+O_\varepsilon(n^{-1/2+\varepsilon}).
\] 

The paper is organized as follows. \Cref{sec:relative-primality} records the level notation and proves multiplicativity of the fixed-set indicator. \Cref{sec:mobius-formula} proves the M\"obius formula. \Cref{sec:reduction} reduces the M\"obius formula to a sum over $\V_{i,S}$. \Cref{sec:exponent-sets} studies the sets $E_{i,S}$. \Cref{sec:size-reduced-sum} counts the members of $\V_{i,S}$. \Cref{sec:densities} proves the fixed-set density theorem, and \cref{sec:uniform-approximation} applies it to $S=P_i(n)$.

\section{Mosaic Levels and \texorpdfstring{$i$}{i}-Relative Primality}
\label{sec:relative-primality}

Throughout the paper, $S$ denotes a finite set of primes, and $i$ denotes a nonnegative integer unless otherwise stated.

\begin{definition}
\label{def:levels}
Let $n>1$ have canonical prime factorization $n=\prod_{j=1}^k p_j^{\alpha_j}$. The sets $\Lambda_r(n)$ of primes appearing at level $r$ of the mosaic of $n$ are defined recursively by
$\Lambda_0(n)=\varnothing,\,\Lambda_1(n)=\{p_1,\ldots,p_k\},$ and, for $r\geq 2$,
\[
        \Lambda_r(n)=\bigcup_{j=1}^k \Lambda_{r-1}(\alpha_j).
\]
We also set $\Lambda_r(1)=\varnothing$ for every $r\geq 0$.
\end{definition}

\begin{example}
For $n=995328=2^{12}3^5=2^{2^23}3^5$, the first three levels are
\[
        \Lambda_1(n)=\{2,3\},\qquad\Lambda_2(n)=\{2,3,5\},\qquad\Lambda_3(n)=\{2\}.
\]
\end{example}

\begin{definition}
\label{def:P_i}
For $n\in\Z^+$, let $P_i(n)$ be the set of primes appearing in the first $i$ levels of the mosaic of $n$:
\[
        P_i(n)=\bigcup_{r=1}^i \Lambda_r(n).
\]
Equivalently, $P_0(n)=\varnothing$, $P_i(1)=\varnothing$, and for $i\geq 1$,
\[
        P_i(n)=\{p_1,\ldots,p_k\}\cup \bigcup_{j=1}^k P_{i-1}(\alpha_j),
\]
where $n=\prod_{j=1}^k p_j^{\alpha_j}$ is the canonical prime factorization of $n>1$.
\end{definition}

\begin{definition}
\label{def:i-relative-prime}
Two positive integers $a$ and $b$ are called \emph{$i$-relatively prime} if
\[
        P_i(a)\cap P_i(b)=\varnothing.
\]
\end{definition}

The next lemma records how $P_i$ behaves on products of coprime integers.

\begin{lemma}
\label{lem:P_i-coprime-product}
If $\gcd(a,b)=1$, then $P_i(ab)=P_i(a)\cup P_i(b)$ for every $i\geq 0$.
\end{lemma}

\begin{proof}
For $i=0$, both sides are empty. Assume $i\geq 1$, and write $a = \prod_{j=1}^{k} p_j^{\alpha_j}$ and $b = \prod_{\ell=1}^{r} q_\ell^{\beta_\ell}$ in canonical prime factorizations. Since $\gcd(a, b) = 1$, the primes $p_1, \dots, p_k$ are distinct from $q_1, \dots, q_r$, so
    \[ab = p_1^{\alpha_1} \cdots p_k^{\alpha_k} q_1^{\beta_1} \cdots q_r^{\beta_r}\]
is the canonical factorization of $ab$. By \cref{def:P_i},
\begin{align*}
        P_i(ab)
        &=\{p_1,\ldots,p_k,q_1,\ldots,q_r\}
        \cup\bigcup_{j=1}^k P_{i-1}(\alpha_j)
        \cup\bigcup_{\ell=1}^r P_{i-1}(\beta_\ell)  \\
        &=P_i(a)\cup P_i(b). \qedhere
\end{align*}
\end{proof}

\begin{definition}
\label{def:chi}
For a finite set of primes $S$, define
\[
        \chi_{i,S}(m)=
        \begin{cases}
        1, & P_i(m)\cap S=\varnothing,\\
        0, & P_i(m)\cap S\neq\varnothing.
        \end{cases}
\]
\end{definition}

\begin{lemma}
\label{lem:chi-multiplicative}
For every finite set of primes $S$, the function $\chi_{i,S}$ is multiplicative.
\end{lemma}

\begin{proof}
Suppose $\gcd(a,b)=1$. By \cref{lem:P_i-coprime-product},
\[
        P_i(ab)\cap S=\bigl(P_i(a)\cap S\bigr)\cup\bigl(P_i(b)\cap S\bigr).
\]
Thus $P_i(ab)\cap S=\varnothing$ if and only if both $P_i(a)\cap S$ and $P_i(b)\cap S$ are empty. Hence $\chi_{i,S}(ab)=\chi_{i,S}(a)\chi_{i,S}(b),$ and $\chi_{i,S}(1)=1$ since $P_i(1)=\varnothing$.
\end{proof}

\begin{lemma}
\label{lem:chi-prime-power}
For $i\geq 1$, prime $p$, and integer $k\geq 1$,
\[
        \chi_{i,S}(p^k)=
        \begin{cases}
        0, & p\in S,\\
        \chi_{i-1,S}(k), & p\notin S.
        \end{cases}
\]
\end{lemma}
\begin{proof}
By \cref{def:P_i},
\[
        P_i(p^k)=\{p\}\cup P_{i-1}(k).
\]
If $p\in S$, then $P_i(p^k)\cap S\neq \varnothing$, so $\chi_{i,S}(p^k)=0$. If $p\notin S$, then
\[
        P_i(p^k)\cap S=P_{i-1}(k)\cap S,
\]
and the claim follows.
\end{proof}


\section{A M\"obius Formula for the Level Totient}
\label{sec:mobius-formula}
We now define the $i$-level totient function $\varphi_i(n)$ and write it as a divisor sum. Since the fixed-set indicator $\chi_{i,S}$ is multiplicative, its M\"obius transform has prime-power values that can be computed explicitly. The rest of the paper studies the nonzero terms of this transform.

\begin{definition}
\label{def:level-totient}
For $n\in\Z^+$ and $i\geq 0$, define the \emph{$i$-level totient} by
\[
        \varphi_i(n)=\sum_{m=1}^n \chi_{i,P_i(n)}(m)=\sum_{m\leq n}\1\{P_i(m)\cap P_i(n)=\varnothing\}.
\]
\end{definition}

For $i=1$, this is Euler's totient function. For $i\geq 2$, it need not be multiplicative as a function of $n$. For example, one computes $\varphi_i(72)=22$, while $\varphi_i(8)\varphi_i(9)=9$ for every $i\geq 2$.

\begin{definition}
\label{def:h}
Let $h_{i,S}=\chi_{i,S}*\mu$, where $\mu$ is the ordinary M\"obius function. Equivalently,
\[
        h_{i,S}(n)=\sum_{d\mid n}\chi_{i,S}(d)\mu\left(\frac nd\right).
\]
\end{definition}

As the Dirichlet convolution of the multiplicative functions $\chi_{i,S}$ (\cref{lem:chi-multiplicative}) and $\mu$, the function $h_{i,S}$ is multiplicative.

\begin{theorem}
\label{thm:mobius-formula}
For every $n\in\Z^+$ and $i\geq 1$,
\[
        \varphi_i(n)=\sum_{d=1}^n h_{i,P_i(n)}(d)\left\lfloor\frac nd\right\rfloor.
\]
\end{theorem}

\begin{proof}
Let $S=P_i(n)$. Since $h_{i,S}=\chi_{i,S}*\mu$, M\"obius inversion gives $\chi_{i,S}(m)=\sum_{d\mid m}h_{i,S}(d)$. Substituting into \cref{def:level-totient} and writing $m=dk$,
\[
        \varphi_i(n)=\sum_{m\le n}\sum_{d\mid m}h_{i,S}(d)=\sum_{d\le n}h_{i,S}(d)\sum_{k\leq n/d}1=\sum_{d\le n}h_{i,S}(d)\left\lfloor\frac nd\right\rfloor. \qedhere
\]
\end{proof}

We finish this section by proving a property of $h_{i,S}$ on prime powers that will be used for the reduction of the M\"obius sum in \cref{sec:reduction}.

\begin{lemma}
\label{lem:h-prime-power}
For $i\geq 1$, prime $p$, and integer $k\geq 1$,
\[
        h_{i,S}(p^k)=
        \begin{cases}
        -1, & p\in S \text{ and } k=1,\\
        0, & p\in S \text{ and } k\ge2,\\
        0, & p\notin S \text{ and } k=1,\\
        \chi_{i-1,S}(k)-\chi_{i-1,S}(k-1), & p\notin S \text{ and } k\ge2.
        \end{cases}
\]
\end{lemma}

\begin{proof}
By definition,
\[
        h_{i,S}(p^k)=\sum_{j=0}^k \chi_{i,S}(p^j)\mu(p^{k-j}).
\]
Since $\mu(p^r)=0$ for $r\geq 2$, only the terms $j=k$ and $j=k-1$ can contribute. Thus
\[
        h_{i,S}(p^k)=\chi_{i,S}(p^k)-\chi_{i,S}(p^{k-1}).
\]
For $k=1$, this is $\chi_{i,S}(p)-1$, which equals $-1$ if $p\in S$ and $0$ otherwise. For $k\geq 2$, \cref{lem:chi-prime-power} gives $0$ when $p\in S$, and gives
\[
        \chi_{i-1,S}(k)-\chi_{i-1,S}(k-1)
\]
when $p\notin S$.
\end{proof}

\section{Reduction of the M\"obius Sum}
\label{sec:reduction}

We now identify the nonzero terms in the M\"obius formula. By \cref{lem:h-prime-power} the primes in $S$ that contribute nonzero terms occur only to the first power, while primes outside $S$ can occur only with exponents $a\ge2$ at which $\chi_{i-1,S}(a)$ and $\chi_{i-1,S}(a-1)$ differ. We begin by recalling the notion of a powerful integer.

\begin{definition}
\label{def:powerful}
A positive integer $v$ is \emph{powerful} if $p^2\mid v$ for every prime $p\mid v$. By convention, $1$ is powerful.
\end{definition}

Since the prime-power formula in \cref{lem:h-prime-power} depends only on the difference between consecutive values of $\chi_{i-1,S}$ for primes outside $S$, we isolate this difference through the following notation.

For $k\geq 2$, define
\[
        \Delta_{i,S}(k)=\chi_{i-1,S}(k)-\chi_{i-1,S}(k-1),
\]
and set $E_{i,S}=\{k\geq 2:\Delta_{i,S}(k)\ne0\}.$

By \cref{lem:h-prime-power}, if $p\notin S$ and $k\geq 2$, then $h_{i,S}(p^k)=\Delta_{i,S}(k).$

\begin{definition}
\label{def:V}
For a finite set of primes $S$, let $\V_{i,S}$ be the set of positive integers $v$ such that, if
\[
        v=\prod_{j=1}^r p_j^{\alpha_j}
\]
is its canonical prime factorization, then:
\begin{enumerate}[label=(\roman*)]
        \item $p_j\notin S$ for every $j$;
        \item $\alpha_j\in E_{i,S}$ for every $j$.
\end{enumerate}
We include $1$ in $\V_{i,S}$.
\end{definition}

Note that every element of $\V_{i,S}$ is powerful since $E_{i,S}\subseteq\{2,3,4,\ldots\}$.

\begin{theorem}
\label{thm:support-split}
If $h_{i,S}(d)\neq 0$, then $d$ factors uniquely as $d=uv$, where $u\mid R_S$, $R_S=\prod_{p\in S}p$, and $v\in\V_{i,S}$. Moreover,
\[
        h_{i,S}(d)=\mu(u)h_{i,S}(v).
\]
Conversely, if $u\mid R_S$, $v\in\V_{i,S}$, and $\gcd(u,v)=1$, then $h_{i,S}(uv)=\mu(u)h_{i,S}(v)$ and this value is nonzero.
\end{theorem}

\begin{proof}
Write $d=\prod_{p\mid d}p^{k_p}$. Since $h_{i,S}$ is multiplicative,
\[
        h_{i,S}(d)=\prod_{p\mid d}h_{i,S}(p^{k_p}).
\]
If $h_{i,S}(d)\neq 0$, then every local factor is nonzero. By \cref{lem:h-prime-power}, this forces $k_p=1$ for $p\in S$, and $k_p\in E_{i,S}$ for $p\notin S$. Hence, with
\[
        u=\prod_{\substack{p\mid d\\ p\in S}}p,
        \qquad
        v=\prod_{\substack{p\mid d\\ p\notin S}}p^{k_p},
\]
we have $u\mid R_S$, $v\in\V_{i,S}$, and $d=uv$. The factorization is unique because it separates the prime divisors of $d$ according to membership in $S$.

Since $\gcd(u,v)=1$, multiplicativity gives $h_{i,S}(d)=h_{i,S}(u)h_{i,S}(v)$. The integer $u$ is squarefree and all of its prime divisors lie in $S$, so \cref{lem:h-prime-power} gives
\[
        h_{i,S}(u)=\prod_{p\mid u}(-1)=\mu(u).
\]
This proves the first assertion. The converse follows from the same local factor calculation: if $u\mid R_S$ and $v\in\V_{i,S}$, then every local factor in $h_{i,S}(uv)$ is nonzero, and multiplicativity gives the displayed formula.
\end{proof}

\begin{definition}
\label{def:phi}
For a finite set of primes $S$, define
\[
        \Phi_S(x)=\sum_{u\mid R_S}\mu(u)\left\lfloor\frac xu\right\rfloor,
        \qquad
        R_S=\prod_{p\in S}p.
\]
Equivalently, $\Phi_S(x)$ is the number of positive integers at most $x$ with no prime factor in $S$.
\end{definition}

\begin{theorem}
\label{thm:reduced-formula}
For every $n\in\Z^+$ and $i\geq 1$, writing $S=P_i(n)$, we have
\[
        \varphi_i(n)=\sum_{\substack{v\in\V_{i,S}\\ v\leq n}}h_{i,S}(v)\Phi_S\left(\left\lfloor\frac nv\right\rfloor\right).
\]
\end{theorem}

\begin{proof}
By \cref{thm:mobius-formula}, $\varphi_i(n)=\sum_{d\leq n}h_{i,S}(d)\lfloor n/d\rfloor$.
By \cref{thm:support-split}, each nonzero term factors uniquely as $d=uv$ with
$u\mid R_S$, $v\in\V_{i,S}$, and $h_{i,S}(d)=\mu(u)h_{i,S}(v)$, so
\[
        \varphi_i(n)=\sum_{\substack{v\in\V_{i,S}\\ v\leq n}}h_{i,S}(v)\sum_{\substack{u\mid R_S\\ uv\leq n}}\mu(u)\left\lfloor\frac n{uv}\right\rfloor.
\]
For fixed $v$ the terms with $uv>n$ are zero, and $\lfloor n/(uv)\rfloor=\lfloor\lfloor n/v\rfloor/u\rfloor$, so the inner sum equals
\[
\sum_{u\mid R_S}\mu(u)\lfloor\lfloor n/v\rfloor/u\rfloor=\Phi_S(\lfloor n/v\rfloor). \qedhere
\]
\end{proof}

\begin{corollary}
\label{cor:i1-euler}
For $i=1$, writing $S=P_1(n)$,
\[
        \varphi_1(n)=\Phi_S(n)=\varphi(n).
\]
\end{corollary}

\begin{proof}
Since $P_0(m)=\varnothing$ for every $m$, we have $\chi_{0,S}(m)=1$ for all $m$. Hence $\Delta_{1,S}(k)=0$ for every $k\geq 2$, so $E_{1,S}=\varnothing$ and $\V_{1,S}=\{1\}$. The reduced formula gives
\[
        \varphi_1(n)=h_{1,S}(1)\Phi_S(n)=\Phi_S(n).
\]
Since $S=P_1(n)$ is the set of prime divisors of $n$, $\Phi_S(n)$ counts the positive integers at most $n$ that are relatively prime to $n$, namely $\varphi(n)$.
\end{proof}


\section{The Exponent Sets}
\label{sec:exponent-sets}

The reduced formula depends on the exponent set $E_{i,S}$. We prove two structural results: the first identifies the smallest possible exponent for arbitrary $i \geq 2$, and the second describes $E_{2,S}$ through congruence conditions modulo $R_S = \prod_{p \in S}p$.

\begin{proposition}
\label{prop:min-exponent}
    For $S \neq \varnothing$ and $i \geq 2$, the smallest element of $E_{i,S}$ is $q =\min S$.
\end{proposition}
\begin{proof}
We first prove, by induction on $j$, that $P_j(m)\cap S=\varnothing$ for every positive integer $m<q$. 

The case $j = 0$ is immediate from $P_0(m)=\varnothing$. Suppose the claim holds for $j=a$. Let $m<q$ and write $m=\prod_{\ell=1}^{r}p_\ell^{\alpha_\ell}$ in canonical prime factorization. Each prime divisor $p_\ell$ satisfies $p_\ell<q$, so $p_\ell\notin S$. Also, since $p_\ell^{\alpha_\ell}\leq m<q$, we have $\alpha_\ell<q$. By the inductive hypothesis, $P_a(\alpha_\ell)\cap S=\varnothing$ for every $\ell$. Hence, by \cref{def:P_i},
\[ 
P_{a+1}(m) \cap S = \left( \{p_1,\ldots,p_r\} \cup \bigcup_{\ell=1}^{r}P_a(\alpha_\ell)\right) \cap S = \varnothing. 
\]
Now we show $q = \min E_{i,S}$. Suppose $2\leq t<q$. Since both $t$ and $t-1$ are less than $q$, the induction claim gives $\chi_{i-1,S}(t)=\chi_{i-1,S}(t-1)=1.$ Hence $\Delta_{i,S}(t)=0$, so no $t$ with $2\leq t<q$ lies in $E_{i,S}$. It remains to show that $q$ lies in $E_{i,S}$. Since $i\geq 2$, \cref{def:P_i} gives
\[
P_{i-1}(q)=\{q\}\cup P_{i-2}(1)=\{q\}.
\]
Thus $\chi_{i-1,S}(q)=0$. Since $q-1<q$, the induction claim gives $\chi_{i-1,S}(q-1)=1.$ Therefore
\[ 
\Delta_{i,S}(q) = \chi_{i-1,S}(q)-\chi_{i-1,S}(q-1) = -1 \neq 0.
\]
Thus $q\in E_{i,S}$, while no smaller integer $t\geq 2$ lies in $E_{i,S}$.
\end{proof}

\begin{theorem}
\label{thm:E2-structure}
For $k\ge2$, we have $k\in E_{2,S}$ if and only if exactly one of $k$ and $k-1$ is
coprime to $R_S$. In particular, membership depends only on $k\bmod R_S$, so
$E_{2,S}$ is periodic modulo $R_S$.
\end{theorem}

\begin{proof}
By definition, $k\in E_{2,S}$ if and only if $\chi_{1,S}(k)\ne\chi_{1,S}(k-1)$. Since $P_1(m)$
is the set of prime divisors of $m$, $\chi_{1,S}(m)=1 \iff\gcd(m,R_S)=1$. Hence,
the condition holds if and only if exactly one of $k,k-1$ is coprime to $R_S$. Both
coprimality conditions depend only on $k\bmod R_S$, which gives the periodicity.
\end{proof}

\begin{corollary}
\label{cor:E2-density}
The natural density of $E_{2,S}$ is
\[
        2\left(\prod_{p\in S}\Bigl(1-\tfrac1p\Bigr)-\prod_{p\in S}\Bigl(1-\tfrac2p\Bigr)\right).
\]
\end{corollary}

\begin{proof}
By \cref{thm:E2-structure}, we count residue classes $a\bmod R_S$ for which exactly one of $a,a-1$ is coprime to $R_S$. The number of classes with $a$ coprime to $R_S$ is $\varphi(R_S)=\prod_{p\in S}(p-1)$, and, by the Chinese remainder theorem, the number with both $a$ and $a-1$ coprime to $R_S$ is $\prod_{p\in S}(p-2)$, since modulo each $p\in S$ the class of $a$ must avoid $0$ and $1$. Hence 
\[
\prod_{p\in S}(p-1)-\prod_{p\in S}(p-2)
\] 
classes have $a$ coprime but $a-1$ not, and by symmetry the same number have $a-1$ coprime but $a$ not. Adding the symmetric results and dividing by $R_S=\prod_{p\in S}p$ gives the density.
\end{proof}

\begin{example}
If $S=\{2\}$, then $R_S=2$, and $\chi_{1,S}(k)=1$ exactly when $k$ is odd. Hence $\chi_{1,S}(k)$ differs for consecutive integers, so
\[
        E_{2,\{2\}}=\{2,3,4,\ldots\}.
\]
Thus $\V_{2,\{2\}}$ is the set of odd powerful integers.
\end{example}

\section{Counting the Reduced Sum}
\label{sec:size-reduced-sum}
The reduced formula of \cref{thm:reduced-formula} sums over $\V_{i,S}$ in place of all divisors $d\le n$, so the number of its terms is $|\V_{i,S}\cap[1,n]|$. We estimate this count. Every element of $\V_{i,S}$ is $q$-full, which gives the order of magnitude $x^{1/q}$. We then sharpen it to the asymptotic $C_{i,S}x^{1/q}$ with an explicit positive constant.

\begin{proposition}
\label{prop:V-upper}
For $S\neq\varnothing$, $q = \min S$, and $i\geq 2$,
\[
        |\V_{i,S}\cap[1,x]|=O_{i,S}(x^{1/q}).
\]
\end{proposition}
\begin{proof}
By \cref{prop:min-exponent}, every exponent appearing in an element of $\V_{i,S}$ is at least $q$. Thus each element of $\V_{i,S}$ is $q$-full, and the number of $q$-full integers up to $x$ is $O_q(x^{1/q})$ by the Erd\H{o}s--Szekeres estimate \cite{ErdosSzekeres1934}. The claim follows.
\end{proof}

We now record the density of squarefree integers with no prime factor in $S$.

\begin{lemma}
\label{lem:squarefree-asymptotic}
For every finite set of primes \(S\),
\[
        Q_S(y)\sim\frac1{\zeta(2)}\prod_{p\in S}\left(1+\frac1p\right)^{-1}y,
\]
where
\[
        Q_S(y)=\sum_{\substack{m\le y\\ \mu^2(m)=1\\ (m,R_S)=1}}1.
\]
\end{lemma}
\begin{proof}
Let $a_S(m)=\mathbf 1\{\mu^2(m)=1,\ (m,R_S)=1\}$ and note that $a_S$ is multiplicative. For $\Re s>1$, its Dirichlet series is
\begin{align*}
    \sum_{m=1}^{\infty}\frac{a_S(m)}{m^s} 
    &= \prod_p\left(\sum_{e=0}^{\infty}\frac{a_S(p^e)}{p^{es}}\right) \\
    &= \prod_{p\in S}1 \prod_{p\notin S}\left(1+p^{-s}\right) =\frac{\zeta(s)}{\zeta(2s)}\prod_{p\in S}\left(1+p^{-s}\right)^{-1}.
\end{align*}
This Dirichlet series has a simple pole at $s=1$, with residue
\[
        \frac1{\zeta(2)}\prod_{p\in S}\left(1+\frac1p\right)^{-1}.
\]
The remaining factor is holomorphic in a half-plane $\Re s>1-\eta$ for some $\eta>0$. Since the coefficients $a_S(m)$ are nonnegative, the Ikehara--Wiener theorem gives
\[
        Q_S(y)=\sum_{m\le y}a_S(m)\sim\frac1{\zeta(2)}\prod_{p\in S}\left(1+\frac1p\right)^{-1}y. \qedhere
\]
\end{proof}

\begin{theorem}
\label{thm:V-asymptotic}
Let \(S\neq \varnothing\), let \(i\geq 2\), and let \(q=\min S\). Then
\[
        |\V_{i,S}\cap[1,x]|\sim C_{i,S}x^{1/q},
\]
where
\[
        C_{i,S}=\frac{1}{\zeta(2)}\prod_{p\in S}\left(1+\frac1p\right)^{-1}\prod_{p\notin S}\frac{1+\sum_{k\in E_{i,S}}p^{-k/q}}{1+p^{-1}}.
\]
In particular, \(C_{i,S}>0\).
\end{theorem}
\begin{proof}
By \cref{prop:min-exponent}, $q=\min E_{i,S}$, and the Dirichlet series of
$\V_{i,S}$ factors as
\[
        F_{i,S}(s)=\sum_{v\in\V_{i,S}}v^{-s}=\prod_{p\notin S}\Bigl(1+\sum_{k\in E_{i,S}}p^{-ks}\Bigr)=B_S(s)\,H_{i,S}(s),
\]
where
\[
        B_S(s)=\prod_{p\notin S}(1+p^{-qs})=\sum_{\substack{m\geq1\\\mu^2(m)=1\\(m,R_S)=1}}m^{-qs},
        \qquad
        H_{i,S}(s)=\prod_{p\notin S}\frac{1+\sum_{k\in E_{i,S}}p^{-ks}}{1+p^{-qs}}.
\]
We first justify the Dirichlet expansion of $H_{i,S}$ at $s=1/q$. For $p\notin S$, write the local factor
\[
        H_p(s)=\frac{1+\sum_{k\in E_{i,S}}p^{-ks}}{1+p^{-qs}}.
\]
Since $q=\min E_{i,S}$,
\[
        H_p(s)=1+\frac{\sum_{\substack{k\in E_{i,S}\\ k>q}}p^{-ks}}{1+p^{-qs}}.
\]
The expansion $(1+p^{-qs})^{-1}=\sum_{r\geq0}(-1)^r p^{-qrs}$ converges absolutely at $\sigma=\Re(s)=1/q$, so the sum of the absolute values of the coefficients of $H_p(s)-1$ at $\sigma=1/q$ is at most
\[
        \Bigl(\sum_{\substack{k\in E_{i,S}\\ k>q}}p^{-k/q}\Bigr)\Bigl(\sum_{r\geq0}p^{-r}\Bigr)\leq\frac{\sum_{k\geq q+1}p^{-k/q}}{1-p^{-1}} = O_q \bigl(p^{-(q+1)/q}\bigr).
\]
Since $\sum_p p^{-(q+1)/q}<\infty$, the Euler product $H_{i,S}=\prod_{p\notin S}H_p$ converges absolutely at $s=1/q$, with $H_{i,S}(1/q)>0$ because each $H_p(1/q)>0$. Its Dirichlet expansion therefore satisfies
\[
        H_{i,S}(s)=\sum_{b\geq1}\frac{c_b}{b^s},
        \qquad
        \sum_{b\geq1}\frac{|c_b|}{b^{1/q}}<\infty.
\]
Expanding $F_{i,S}=B_S H_{i,S}$ and summing the coefficients with $bm^q\leq x$,
\[
        |\V_{i,S}\cap[1,x]|=\sum_{b\geq1}c_b\,Q_S\bigl((x/b)^{1/q}\bigr).
\]
By \cref{lem:squarefree-asymptotic}, $x^{-1/q}Q_S((x/b)^{1/q})\to\kappa_S b^{-1/q}$ and is at most $b^{-1/q}$, where
\[
        \kappa_S=\frac{1}{\zeta(2)}\prod_{p\in S}(1+1/p)^{-1}.
\]
Since $\sum_b|c_b|b^{-1/q}<\infty$, dominated convergence gives
\[
        |\V_{i,S}\cap[1,x]|\sim\kappa_S H_{i,S}(1/q)\,x^{1/q}.
\]
At $s=1/q$ one has $p^{-qs}=p^{-1}$, so $\kappa_S H_{i,S}(1/q)=C_{i,S}$, and $C_{i,S}>0$ because $\kappa_S>0$ and $H_{i,S}(1/q)>0$.
\end{proof}

\begin{corollary}
\label{cor:term-count}
For fixed $S \neq \varnothing$, $i \geq 2$, and $q = \min S$, the reduced formula in \cref{thm:reduced-formula} has $\sim C_{i,S}n^{1/q}$ nonzero summands.
\end{corollary}
\begin{proof}
Every $v\in\V_{i,S}$ with $v\leq n$ gives a nonzero term of the reduced formula, since $h_{i,S}(v)\neq 0$ and $\Phi_S(\lfloor n/v\rfloor)\geq 1$. By \cref{thm:V-asymptotic} the formula therefore has $\sim C_{i,S}n^{1/q}$ terms.
\end{proof}

\section{Densities for Fixed Prime Sets}
\label{sec:densities}

Fix a finite set of primes $S$. We now study the density of positive integers
whose first $i$ mosaic levels avoid $S$.

\begin{definition}
\label{def:A}
For $N\in\Z^+$, define
\[
        A_{i,S}(N)=\sum_{m\leq N}\chi_{i,S}(m).
\]
When the limit exists, define
\[
        \delta_i(S)=\lim_{N\to\infty}\frac{A_{i,S}(N)}N.
\]
\end{definition}

We first show the limit exists and evaluate it through the M\"obius transform
$h_{i,S}$.

\begin{lemma}
\label{lem:density-series}
For every finite set of primes $S$ and every $i\geq 1$, the density $\delta_i(S)$ exists and satisfies
\[
        \delta_i(S)=\sum_{d=1}^{\infty}\frac{h_{i,S}(d)}{d},
\]
with absolute convergence.
\end{lemma}

\begin{proof}
By M\"obius inversion of $\chi_{i,S}(m)$ we have $A_{i,S}(N)=\sum_{d\leq N}h_{i,S}(d)\lfloor N/d\rfloor$. By \cref{lem:h-prime-power} the prime-power values of $h_{i,S}$ are at most $1$ in absolute value, hence $|h_{i,S}(d)|\leq1$. By \cref{thm:support-split} every $d$ with $h_{i,S}(d)\neq0$ is $d=uv$ with $u\mid R_S$ and $v\in\V_{i,S}$. Therefore
\[
        \sum_{d\geq1}\frac{|h_{i,S}(d)|}{d} \leq\left(\sum_{u\mid R_S}\frac{1}{u}\right)\sum_{v\in\V_{i,S}}\frac{1}{v}<\infty,
\]
since $\V_{i,S}$ consists of powerful integers and $\sum_{v\text{ powerful}}1/v=\zeta(2)\zeta(3)/\zeta(6)$ \cite{Golomb1970}. Writing $\lfloor N/d\rfloor=N/d-\theta_{N,d}$ with $0\leq\theta_{N,d}<1$,
\[
        \frac{A_{i,S}(N)}N=\sum_{d\leq N}\frac{h_{i,S}(d)}d-\frac{1}{N}\sum_{d\leq N}h_{i,S}(d)\theta_{N,d}.
\]
The first sum converges to $\sum_{d\ge1}h_{i,S}(d)/d$. The second is at most
\[
        \frac{1}{N}\sum_{d\leq N}|h_{i,S}(d)|,
\] 
which tends to $0$ by Kronecker's lemma because $\sum_d|h_{i,S}(d)|/d$ converges. Hence $\delta_i(S)$ exists and equals $\sum_{d\ge1}h_{i,S}(d)/d$.
\end{proof}

Since $h_{i,S}$ is multiplicative, the series of \cref{lem:density-series} factors
as an Euler product, which we now compute.

\begin{theorem}
\label{thm:delta-euler-product}
For every finite set of primes $S$ and every $i\geq 1$,
\[
        \delta_i(S)=\prod_{p\in S}\left(1-\frac1p\right)\prod_{p\notin S}\left(1-\frac1p\right)\left(1+\sum_{k=1}^{\infty}\frac{\chi_{i-1,S}(k)}{p^k}\right).
\]
\end{theorem}
\begin{proof}
By \cref{lem:density-series} and multiplicativity of $h_{i,S}$, it follows that $\delta_i(S)$ has the Euler product
\[
        \delta_i(S)=\prod_p\left(1+\sum_{k=1}^{\infty}\frac{h_{i,S}(p^k)}{p^k}\right).
\]
For $p\in S$, \cref{lem:h-prime-power} gives $h_{i,S}(p)=-1$ and $h_{i,S}(p^k)=0$ for $k\geq 2$, so the local factor is $1-1/p$. For $p\notin S$, \cref{lem:h-prime-power} gives $h_{i,S}(p)=0$ and $h_{i,S}(p^k)=\Delta_{i,S}(k)$ for $k\geq 2$. Thus, the local factor is
\[
        1+\sum_{k=2}^{\infty}\frac{\Delta_{i,S}(k)}{p^k}.
\]
Since $\Delta_{i,S}(k) = \chi_{i-1,S}(k)-\chi_{i-1,S}(k-1)$ and $\chi_{i-1,S}(1)=1$, we have
\begin{align*}
        1+\sum_{k=2}^{\infty}\frac{\Delta_{i,S}(k)}{p^k}
        &=1+\left(\sum_{k=1}^{\infty}\frac{\chi_{i-1,S}(k)}{p^k}-\frac{1}{p}\right)-\frac1p\sum_{k=1}^{\infty}\frac{\chi_{i-1,S}(k)}{p^k} \\
        &=\left(1-\frac{1}{p}\right)\left(1+\sum_{k=1}^{\infty}\frac{\chi_{i-1,S}(k)}{p^k}\right).
\end{align*}
Combining the local factors gives the stated product.
\end{proof}

We finish the section by recovering the usual density of integers avoiding a fixed set of primes and giving an explicit formula for the Euler product for $i = 2$. 

\begin{corollary}
\label{cor:delta-legendre}
For every finite set of primes $S$,
\[
        \delta_1(S)=\prod_{p\in S}\left(1-\frac1p\right).
\]
\end{corollary}
\begin{proof}
Taking $i=1$ in \cref{thm:delta-euler-product}, the local factor for $p\notin S$ is
\[
        \left(1-\frac1p\right)\left(1+\sum_{k=1}^{\infty}\frac{\chi_{0,S}(k)}{p^k}\right).
\]
Since $\chi_{0,S}(k)=1$ for every $k\in\mathbb Z^+$, we have
\[
        \delta_1(S) = \prod_{p\in S}\left(1-\frac1p\right)\prod_{p\notin S}\left(1-\frac1p\right)\left(1+\frac1{p-1}\right)= \prod_{p\in S}\left(1-\frac1p\right).\qedhere
\]
\end{proof}

\begin{corollary}
\label{cor:delta-two-formula}
Let $R_S=\prod_{q\in S}q$. Then
\[
        \delta_2(S) = \prod_{p\in S}\left(1-\frac{1}{p}\right) \prod_{p\notin S} \left(1-\frac{1}{p}\right) \left(1+\sum_{d\mid R_S}\frac{\mu(d)}{p^d-1}\right).
\]
\end{corollary}
\begin{proof}
Taking $i=2$ in \cref{thm:delta-euler-product}, it remains to compute $\sum_{k=1}^{\infty}\chi_{1,S}(k)/p^k$. Since $\chi_{1,S}(k)=1$ if and only if $\gcd(k,R_S)=1$, we have
\[
        \chi_{1,S}(k) = \sum_{d\mid\gcd(k,R_S)}\mu(d).
\]
Substituting and switching the order of summation gives
\[
        \sum_{k=1}^{\infty}\frac{\chi_{1,S}(k)}{p^k} = \sum_{d\mid R_S}\mu(d) \sum_{\substack{k\geq 1\\d\mid k}}\frac1{p^k}.
\]
For fixed $d$, the inner sum is $\sum_{j=1}^{\infty}p^{-dj}=1/(p^d-1)$. Hence
\[
        \sum_{k=1}^{\infty}\frac{\chi_{1,S}(k)}{p^k} = \sum_{d\mid R_S}\frac{\mu(d)}{p^d-1}.
\]
Substituting into \cref{thm:delta-euler-product} gives the result.
\end{proof}

\begin{lemma}
\label{lem:truncation}
Let $0<\theta<1$, let $S$ be a finite set of primes, and let $i\geq1$. If
$|\V_{i,S}\cap[1,X]|=O(X^{\theta})$, then
\[
        A_{i,S}(N)=\delta_i(S)N+O_\theta\Bigl(\sum_{u\mid R_S}(N/u)^{\theta}\Bigr).
\]
\end{lemma}
\begin{proof}
By \cref{lem:density-series} and $\lfloor N/d\rfloor=N/d+O(1)$,
\[
        |A_{i,S}(N)-\delta_i(S)N|\le\sum_{d\le N}|h_{i,S}(d)|+N\sum_{d>N}\frac{|h_{i,S}(d)|}{d}.
\]
By \cref{thm:support-split} every $d$ with $h_{i,S}(d)\neq0$ is $uv$ with $u\mid R_S$, $v\in\V_{i,S}$, and $|h_{i,S}(d)|\leq1$, so the right-hand side is at most
\[
        \sum_{u\mid R_S}\Bigl(|\V_{i,S}\cap[1,N/u]|+\frac {N}{u}\sum_{\substack{v\in\V_{i,S}\\ v>N/u}}\frac{1}{v}\Bigr).
\]
The hypothesis with $X=N/u$ gives $|\V_{i,S}\cap[1,N/u]|=O((N/u)^{\theta})$, and partial summation gives 
\[
        \sum_{\substack{v\in\V_{i,S}\\ v>N/u}}\frac{1}{v}=O\left((N/u)^{\theta-1}\right),
\]
hence each summand is $O_\theta((N/u)^{\theta})$.
\end{proof}

\begin{theorem}
\label{thm:quantitative-density}
Let $S\neq \varnothing$, let $i\geq 2$, and let $q=\min S$. Then
\[
        A_{i,S}(N) = \delta_i(S)N+O_{i,S}(N^{1/q}).
\]
\end{theorem}

\begin{proof}
By \cref{prop:V-upper}, $|\V_{i,S}\cap[1,X]|=O_{i,S}(X^{1/q})$, so
\cref{lem:truncation} gives
\[
        A_{i,S}(N)=\delta_i(S)N+O_{i,S}\bigl(\sum_{u\mid R_S}(N/u)^{1/q}\bigr)=\delta_i(S)N+O_{i,S}(N^{1/q}),
\]
since $\sum_{u\mid R_S}u^{-1/q}$ is constant for fixed $S$.
\end{proof}

\section{A Uniform Approximation for the Level Totient}
\label{sec:uniform-approximation}

The preceding section studies the density $\delta_i(S)$ for a fixed set $S$. To estimate $\varphi_i(n)/n$, we set $S = P_i(n)$. Since this set varies with $n$, we prove an estimate uniform in $n$.

\begin{lemma}
\label{lem:radical-bound}
For every $i\geq 0$ and every $n\in\Z^+$,
\[
        \prod_{p\in P_i(n)}p\leq n.
\]
\end{lemma}

\begin{proof}
We argue by induction on $i$. For $i=0$, the product is empty and equals $1$, so the result is immediate. For $i\ge1$ write $n=\prod_{j=1}^k p_j^{\alpha_j}$. By \cref{def:P_i},
\[
        P_i(n)=\{p_1,\ldots,p_k\}\cup\bigcup_{j=1}^k P_{i-1}(\alpha_j).
\]
By the induction hypothesis $\prod_{p\in P_{i-1}(\alpha_j)}p\le\alpha_j$, and $\alpha_j\leq p_j^{\alpha_j-1}$ for every prime $p_j$ and integer $\alpha_j\geq1$, so
\[
        \prod_{p\in P_i(n)}p
        \leq \left(\prod_{j=1}^k p_j\right)\left(\prod_{j=1}^k\prod_{p\in P_{i-1}(\alpha_j)}p\right)
        \leq\prod_{j=1}^k p_j\alpha_j
        \leq\prod_{j=1}^k p_j^{\alpha_j}=n. \qedhere
\]
\end{proof}

\begin{lemma}
\label{lem:small-prime-product}
For every $\varepsilon>0$,
\[
        \prod_{p\in S}\left(1+p^{-1/2}\right)=O_\varepsilon\left(\prod_{p\in S}p^\varepsilon\right),
\]
uniformly over all finite sets of primes $S$.
\end{lemma}

\begin{proof}
Since
\[
        \frac{\log(1+p^{-1/2})}{\log p}\to0
\]
as $p\to\infty$, there exists $M_\varepsilon$ such that $1+p^{-1/2}\leq p^\varepsilon$ for every prime $p>M_\varepsilon$. Hence
\[
        \prod_{p\in S}(1+p^{-1/2})\leq\prod_{p\leq M_\varepsilon}(1+p^{-1/2})\prod_{\substack{p\in S\\ p>M_\varepsilon}}p^\varepsilon\leq C_\varepsilon\prod_{p\in S}p^\varepsilon,
\]
where $C_\varepsilon=\prod_{p\leq M_\varepsilon}(1+p^{-1/2})$.
\end{proof}

\begin{theorem}
\label{thm:uniform-approximation}
For every $i\geq 1$ and every $\varepsilon>0$,
\[
        \frac{\varphi_i(n)}n=\delta_i(P_i(n))+O_{\varepsilon}(n^{-1/2+\varepsilon}).
\]
\end{theorem}

\begin{proof}
Set $S=P_i(n)$, so $\varphi_i(n)=A_{i,S}(n)$. Every element of $\V_{i,S}$ is
powerful, so $|\V_{i,S}\cap[1,X]|=O(X^{1/2})$ \cite{Golomb1970}. By
\cref{lem:truncation},
\begin{align*}
        \varphi_i(n)
        &=n\,\delta_i(P_i(n)) +O\Bigl(\sum_{u\mid R_S}(n/u)^{1/2}\Bigr) \\
        &=n\,\delta_i(P_i(n))+O\Bigl(n^{1/2}\prod_{p\in S}(1+p^{-1/2})\Bigr).
\end{align*}
By \cref{lem:small-prime-product}, $\prod_{p\in S}(1+p^{-1/2})=O_\varepsilon(\prod_{p\in S}p^{\varepsilon})$, and \cref{lem:radical-bound} gives $\prod_{p\in S}p\le n$. Hence, the error is $O_\varepsilon(n^{1/2+\varepsilon})$, and the claim follows.
\end{proof}

\textbf{Acknowledgments.} The third author would like to thank the Hulbert Taft Jr. Library at The Taft School, particularly Mr. Sean Padgett, for graciously providing him with full access to the library. 

\textbf{AI Use Statement.} No AI tools were used in  this paper.

\printbibliography

\end{document}